\documentclass{amsart}
\usepackage{amssymb,latexsym}
\usepackage{amscd,amsthm}

\usepackage[all]{xy}
\usepackage{tikz}

\usepackage{tikz-cd}

\usepackage{extpfeil}
\newtheorem{thm}{\bf Theorem}[section]
\newtheorem{cor}[thm]{\bf Corollary}

\newtheorem{lem}[thm]{\bf Lemma}
\newtheorem{prop}[thm]{\bf Proposition}

\newtheorem{rem}[thm]{\bf Remark}

\DeclareMathOperator{\Ext}{Ext}
\DeclareMathOperator{\Hom}{Hom}

\DeclareMathOperator{\Qcoh}{Qcoh}

\newcommand{\cqfd}
{\hspace{1cm}
\rule{2mm}{2mm}%
\medbreak%
\par%
}

\def\fd{{\rm fd}}
\def\id{{\rm id}}

\def\Gfd{{\rm Gfd}}
\def\Gid{{\rm Gid}}

\def\resdim{{\rm resdim}}

\def\flat{\rm Flat}
\def\cot{\rm Cot}
\def\inj{\rm Inj}

\def\Q{\mathcal{Q}}
\def\W{\mathcal{W}}

\def\R{\mathcal{R}}

\def\M{\mathcal{M}}

\def\A{\mathcal{A}}

\def\X{\mathcal{X}}
\def\Y{\mathcal{Y}}
\def\O{\mathcal{O}}

\def\GF{\rm GFlat}

\def\GI{\rm GInj}

\def\Qcoh{\rm Qcoh}

\def\Ker{{\rm Ker}}

\def\Ho{{\rm Ho}}

\def\Ext{{\rm Ext}}

\def\Hom{{\rm Hom}}

\begin{document}

\title[Approximations and Hovey triples by objects of finite homological dimensions]{Approximations and Hovey triples by objects of finite homological dimensions: Applications to sheaves}

\author{Rachid El Maaouy}
\address{SMARTiLab, Moroccan School of Engineering Sciences (EMSI), Rabat, Morocco}

\email[Rachid El Maaouy]{elmaaouy.rachid@gmail.com}

\author{Hanane Ouberka}
 
\address{Laboratoire Interdisciplinaire de Modélisation Simulation et Innovation Scientifique (LIMSIS), Département des sciences École normale supérieure, Moulay Ismaïl University of Meknès}

\email[Hanane Ouberka]{H.ouberka@umi.ac.ma}

\date{\today}

\keywords{(pre)cover, (pre)envelope, cotorsion pair, abelian model structure, Gorenstein sheaves and dimension} 

\thanks{2020 Mathematics Subject Classification. 18N40, 18G20, 18F20.
}

\begin{abstract}
Let $\mathcal{Q}$ be a class of objects in an abelian category $\mathcal{A}$ which need not have enough projective or injective objects. In this paper, we prove that if $\mathcal{Q}$ is the first class of a Hovey triple $(\mathcal{Q},\mathcal{W},\mathcal{R})$ in $\mathcal{A}$ satisfying certain assumptions—weaker than those required in the recent literature—then $\mathcal{Q}_n$, the class of objects with $\mathcal{Q}$-resolution dimension at most an integer $n\ge 0$, forms the first class of a hereditary Hovey triple $\mathcal{M}_n=(\mathcal{Q}_n,\mathcal{W}_{\mathcal{Q},n},\mathcal{R}_{\mathcal{Q},n})$, where $\mathcal{W}_{\mathcal{Q},n}$ and $\mathcal{R}_{\mathcal{Q},n}$ are described explicitly. Consequently, $\mathcal{Q}_n$ is the left-hand side of a complete hereditary cotorsion pair and hence a special precovering class. The dual statement is also established. As a main application, we construct an abelian model structure on $\Qcoh(X)$, the category of quasi-coherent sheaves over a semi-separated Noetherian scheme $X$, in which the cofibrant (resp. fibrant) objects are precisely the sheaves with Gorenstein flat (resp. Gorenstein injective) dimension at most $n$.

\end{abstract}

\maketitle

\section{Introduction} 

In relative homological algebra, it is of fundamental importance to know whether a class of objects in an abelian category is (pre)covering or (pre)enveloping. In homotopical algebra, it is equally important to determine whether such a class is part of a (complete) cotorsion pair or a Hovey triple.

Approximations and complete cotorsion pairs by modules of finite classical (Gorenstein) dimension have been widely studied in the literature over recent years; see \cite[Section 8.1]{GT12} and \cite[Chapters III and IV]{Per16}. For some recent developments in this direction, we refer the reader to \cite{Elm24,Elm25,SWZ25,GLZ25}. The interest in these classes has recently motivated many researchers to investigate the following general problem: when are the homological and homotopical properties of a class $\mathcal{Q}$ in an abelian category inherited by $\mathcal{Q}_n$, the class of all objects with $\mathcal{Q}$-(co)resolution dimension at most an integer $n\ge 0$? For instance, El Maaouy asked in \cite{Elm25}, when the following assertions (and their duals) hold:
\begin{enumerate}
    \item If $\mathcal{Q}$ is (special pre)covering, so is $\mathcal{Q}_n$.
    \item If $\mathcal{Q}$ is the left‑hand part of a (complete) cotorsion pair, so is $\mathcal{Q}_n$.
    \item If $\mathcal{Q}$ is the first component of a Hovey triple, so is $\mathcal{Q}_n$.
\end{enumerate}

For the definitions of approximations, cotorsion pairs, and Hovey triples/abelian model structures, and how they are related to each other, we refer the reader to Subsections \ref{Subsec. approx}, \ref{Subsec. cot-pair}, and \ref{Subsec. Hov-triple}. 

An answer to the above questions is provided in \cite[Theorem 3.12]{Elm25}, which states that when the class $\Q$ is part of a Hovey triple $\M=(\Q,\W,\R)$ satisfying some assumptions, then the class $\Q_n$ is also part of a Hovey triple $\M_n=(\Q_n,\W,\Q_{\W,n}^\perp)$ where $\Q_{\W,n}$ denotes the class of objects with $(\mathcal{Q}\cap \W)$-resolution dimension at most $n\geq 0$. 
As a consequence, the pair $(\Q_n,\W\cap \Q_{\W,n}^\perp)$ is a complete and hereditary cotorsion pair, and $\Q_n$ is a special precovering class. One of the assumptions required in \cite[Theorem 3.12]{Elm25} is that the class $\W$ of  objects is part of a complete cotorsion pair $(^\perp\W,\W)$. Unfortunately, this is not always the case, especially for some Hovey triples constructed using \cite[Theorem 1.1]{Gil15}. In very recent work, Shao, Wang, and Zhang have removed this assumption in \cite[Theorem 3.6]{SWZ25} at the cost of assuming that the category has enough projective and injective objects. However, in algebraic geometry, it is known that categories of sheaves do not, in general, have enough projective objects; thus, it is not possible to apply this last result in such categories.

The goal of this paper is twofold. First, we significantly improve the two results \cite[Theorem A]{Elm25} and \cite[Theorem 3.6]{SWZ25}. In particular, we remove the first assumption of \cite[Theorem A]{Elm25} without requiring that the category $\mathcal{A}$ has enough projective or injective objects, a condition that was essential in \cite[Theorem 3.6]{SWZ25}. As the first main result of this paper, we prove the following theorem.

\bigskip

\noindent\textbf{Theorem A.} Let $(\mathcal{Q},\mathcal{W},\mathcal{R})$ be a hereditary Hovey triple in an abelian category $\mathcal{A}$. For every integer $n\ge 0$, the following assertions hold.
\begin{enumerate}
    \item If $\bigl(\mathcal{Q}_{\mathcal{W},n},\mathcal{Q}_{\mathcal{W},n}^\perp\bigr)$ is a complete cotorsion pair, then
    \[
    \bigl(\mathcal{Q}_n,\ \mathcal{W},\ \mathcal{Q}_{\mathcal{W},n}^\perp\bigr)
    \]
    is a hereditary Hovey triple on $\mathcal{A}$. In this case, $(\mathcal{Q}_n,\mathcal{Q}_n^\perp)$ is a complete hereditary cotorsion pair with $\mathcal{Q}_n^\perp = \mathcal{W} \cap \mathcal{Q}_{\mathcal{W},n}^\perp$, and $\mathcal{Q}_n$ is a special precovering class.
    
    \item If $\bigl({}^\perp\!\mathcal{R}_{\mathcal{W},n},\ \mathcal{R}_{\mathcal{W},n}\bigr)$ is a complete cotorsion pair, then
    \[
    \bigl({}^\perp\!\mathcal{R}_{\mathcal{W},n},\ \mathcal{W},\ \mathcal{R}_n\bigr)
    \]
    is a hereditary Hovey triple on $\mathcal{A}$. In this case, $({}^\perp\!\mathcal{R}_n,\mathcal{R}_n)$ is a complete hereditary cotorsion pair with ${}^\perp\!\mathcal{R}_n = \mathcal{W} \cap {}^\perp\!\mathcal{R}_{\mathcal{W},n}$, and $\mathcal{R}_n$ is a special preenveloping class.
\end{enumerate}

\bigskip

Our second main goal is to provide some applications in $\Qcoh(X)$, the category of quasi-coherent sheaves of $\mathcal{O}_X$-modules, where $(X,\mathcal{O}_X)$ is a ringed space—a category that is abelian but generally lacks enough projective objects. Applying Theorem A to the two abelian model structures recently constructed on $\Qcoh(X)$:
\begin{enumerate}
    \item \textbf{Gorenstein flat model structure} $\mathcal{M}_{\GF}=(\GF(X),\mathcal{W}_{\mathrm{flat}},\flat(X)^\perp)$, constructed by Christensen, Estrada, and Thompson \cite[Theorem 2.5]{CET21}; and
    \item \textbf{Gorenstein injective model structure} $\mathcal{M}_{\GI}=({}^\perp\!\inj(X),\mathcal{W}_{\mathrm{inj}},\GI(X))$, constructed by Estrada and Gillespie \cite[Theorem 5.2]{EG25},
\end{enumerate}
we obtain our second main result.

\bigskip

\noindent\textbf{Theorem B} (Theorems \ref{n-GF-Hovey-triple} and \ref{GIn-cot-pair-Hov-trip})
Let $X$ be a semi‑separated Noetherian scheme. For every integer $n\ge 0$, the following assertions hold on $\Qcoh(X)$:
\begin{enumerate}
    \item There exists a hereditary abelian model structure
          \[
          \mathcal{M}_{\GF_n}= \bigl( \GF_n(X),\ \mathcal{W}_{\mathrm{flat}},\ \flat_n(X)^\perp \bigr).
          \]
          Consequently, $(\GF_n(X),\GF_n(X)^\perp)$ is a complete hereditary cotorsion pair with
          $\GF_n(X)^\perp = \mathcal{W}_{\mathrm{flat}} \cap \flat_n(X)^\perp$, and $\GF_n(X)$ is a special precovering class.
    \item There exists a hereditary abelian model structure
          \[
          \mathcal{M}_{\GI_n}= \bigl( {}^\perp\!\inj_n(X),\ \mathcal{W}_{\mathrm{inj}},\ \GI_n(X) \bigr).
          \]
          Consequently, $({}^\perp\!\GI_n(X),\GI_n(X))$ is a complete hereditary cotorsion pair with
          ${}^\perp\!\GI_n(X) = {}^\perp\!\inj_n(X) \cap \mathcal{W}_{\mathrm{inj}}$, and $\GI_n(X)$ is a special preenveloping class.
\end{enumerate}

The results from Section~\ref{Main result} that are used to prove Theorem B also yield some consequences of independent interest. For instance, we obtain in Proposition~\ref{n-GF sheaf} several characterizations of when a sheaf $M$ has Gorenstein flat dimension at most $n$. This leads, for instance, to a non‑affine answer to the question of when the Gorenstein flat dimension refines the classical flat dimension (see Corollary~\ref{GF-ref}). Parallel results for Gorenstein injective sheaves are also established (see Proposition~\ref{n-GI sheaf} and Corollary~\ref{GI-ref}).

\section{Preliminaries}

 In this section, we recall some necessary notions and definitions. Throughout the paper, $\A$ will denote an abelian category. By the term \textacutedbl subcategory\textgravedbl, we mean a full additive subcategory closed under isomorphisms, and all classes of objects in $\A$ may be regarded as subcategories of $\A$. 
 
 Let $\X$ be a fixed class of objects in $\A$. In what follows, we use the following standard notations:$$\X^\perp=\{A\in\A|\Ext^1_\A(X,A)=0,\forall X\in\X\},$$
$$^\perp\X=\{A\in\A|\Ext^1_\A(A,X)=0,\forall X\in\X\}.$$

A complex $\textbf{X}$ in $\A$ is called $\Hom_\A(\mathcal{X},-)$-exact (resp., $ \Hom_\A(-,\mathcal{X})$-exact) if $\Hom_\A(X,\textbf{X})$ (resp., $\Hom_\A(\textbf{X},X)$) is an exact complex for every object $X\in\mathcal{X}$.

\subsection{(Co)resolution dimensions}\label{Subsec. res-dim}  Recall from \cite{AB89} that an $\X$-resolution of an object $A\in \A$ is an exact sequence 
 $\cdots \to X_1\to X_0\to A\to 0$
with each $X_i\in\X$. An object $A\in\A$ is said to have $\X$-resolution dimension at most an integer $n \geq  0$, $\resdim_\X(A)\leq  n$, if $A$ has a finite $\X$-resolution: 
$$0\to X_n\to \cdots \to X_1\to X_0\to A\to 0.$$
If $n$ is the least non-negative integer for which such a sequence exists, then its $\X$-resolution dimension is precisely $n$. If there is no such $n$, then we define its $\X$-resolution dimension as $\infty$. One can define $\X$-coresolutions and $\X$-coresolution dimensions dually.

Given a triple $(\Q,\W,\R)$ of classes of objects in $\A$, in what follows, we use the notation $\Q_\W:=\Q\cap \W$ and  $\R_\W:=\W\cap \R$. Moreover, for a given integer $n\geq 0$, we denote by $\Q_{\W,n}:=(\Q_{\W})_n$ (resp.,  $\R_{\W,n}:=(\R_{\W})_n$) the class of objects having $\Q_\W$-resolution (resp., $\R_\W$-coresolution) dimension at most $n$.

\subsection{Approximations.}\label{Subsec. approx}

An $\mathcal{X}$-precover of an object $A\in\mathcal{A}$ is a morphism $f:X\to A$ with $X\in\mathcal{X}$ such that $f_*:\Hom_\mathcal{A}(X',X)\to \Hom_\mathcal{A}(X',A)$ is surjective for every $X'\in\mathcal{X}$. An $\mathcal{X}$-precover $f$ is called special if it is an epimorphism and $\Ker f\in\mathcal{X}^\perp$. If every object has a (special )$\mathcal{X}$-precover, then the class $\mathcal{X}$ is said to be (special) precovering. (special) $\mathcal{X}$-preenvelopes can be defined dually. Standard references for approximations in module categories include \cite{EJ00} and \cite{GT12}, while \cite{EO02} and \cite{Per16} serve as good references in abelian categories.

\subsection{Cotorsion pairs} \label{Subsec. cot-pair} A pair $\left(\X,\Y\right)$ of classes of objects in $\A$ is called a cotorsion pair if $\X^\perp=\Y$ and $\X=\,^\perp\Y$. A cotorsion pair $(\X, \Y)$ is complete if 

(a)  $(\X, \Y)$ is left complete, i.e., any object in $\A$ has a special $\X$-precover, and 

(b)  $(\X, \Y)$ is right complete, i.e., any object in $\A$ has a special $\Y$-preenvelope.  

A cotorsion pair $(\X, \Y)$ is hereditary if $\Ext^i_\A(X,Y)=0$ for all $X\in\X$, $Y\in\Y$, and $i\geq 1$. It follows from \cite[Lemma 6.17]{Sto14} that a complete cotorsion pair $(\X,\Y)$ is hereditary if and only if $\X$ is closed under kernels of epimorphisms, and if and only if $\Y$ is closed under cokernels of monomorphisms. 

A class of objects $\X\subseteq \A$ is said to be generating (resp., cogenerating) if every object of $\A$ is a quotient (resp., subobject) of an object from $\X$. For instance, in any complete cotorsion pair $(\X,\Y)$ in $\A$, the class $\X$ is generating, and the class $\Y$ is cogenerating.  For any cotorsion pair $(\X,\Y)$ in $A$,  the class $\X$
is always generating if $\A$ has enough projectives, and the class $\Y$ is always cogenerating if $\A$ has enough  injective objects.

\begin{lem}\label{lem-left-right-complete}
Let $(\X,\Y)$ be a cotorsion pair in $\A$. 
\begin{enumerate}
\item Assume $\Y$ is cogenerating. If  $(\X,\Y)$ is left complete, then so is right complete.
 \item Assume $\X$ is generating. If $(\X,\Y)$ is right complete, then so is left complete. 
\end{enumerate}

\end{lem} 
\proof We will only prove assertion $(1)$ since $(2)$ is the dual.\\
Let $M \in \A$ and consider the following short exact sequences $$0 \to M \to Y_0 \to C_0 \to 0, \text{ and }0 \to Y_1 \to X_1 \to C_0 \to 0.$$
 The first sequence exists such that $Y_0\in \Y$ since $\Y$ is cogenerating, and the second one exists such that $X_1\in \X$ and $Y_1 \in \Y$ since $(\X,\Y)$ is left complete. Now, we consider the following pullback diagam: 
$$
\begin{tikzcd}[column sep=2em, row sep=1em]
  & & 0 \arrow[d] & 0 \arrow[d] & \\[0.3ex]
  & & Y_1 \arrow[d] \arrow[r, equal] & Y_1 \arrow[d] & \\
  0 \arrow[r] & M \arrow[r] \arrow[d, equal] 
              & D \arrow[r, dashed] \arrow[d, dashed] 
              & X_1 \arrow[r] \arrow[d] 
              & 0 \\
  0 \arrow[r] & M \arrow[r] 
              & Y_0 \arrow[r] \arrow[d] 
              & C_0 \arrow[r] \arrow[d] 
              & 0 \\
  & & 0 & 0. &
\end{tikzcd}
$$
The middle column gives $D \in \Y$ since $Y_1 \in \Y$, $Y_0\in \Y$, and $\Y$ is closed under extensions. Thus, $(\X,\Y)$ is right complete.
\cqfd

\begin{lem} \label{lem-intersection-generating}
Let $(\X,\Y)$ be a cotorsion pair and $\W$ a class that is closed under extensions. 

\begin{enumerate}
 \item If $\W$ is cogenerating, $\X \subseteq \W$ and  $(\X,\Y)$ is right complete, then $\W \cap \Y$ is cogenerating.
    \item  If $\W$ is generating, $\Y \subseteq \W$ and  $(\X,\Y)$ is left complete, then $\W \cap \X$ is generating.
\end{enumerate}
    
\end{lem}
\proof  We will only prove assertion $(1)$ since $(2)$ is the dual.\\
Let $M \in \A$ and consider the following short exact sequences $$0 \to M \to W \to C \to 0, \text{ and }0 \to W \to Y \to X\to 0.$$
 The first sequence exists such that $W\in \W$ since $\W$ is cogenerating and the second one exists such that $Y \in \Y$ and $X\in \X$ since $(\X,\Y)$ is right complete. Now, we consider the following pushout diagam: 
 \[
\begin{tikzcd}[column sep=2em, row sep=1em]
&&0\arrow[d]&0\arrow[d]&\\
0 \arrow[r] 
  & M \arrow[r] \arrow[d, equal] 
  & W \arrow[r] \arrow[d] 
  & C \arrow[r] \arrow[d, dashed] 
  & 0 \\
0 \arrow[r] 
  & M \arrow[r] 
  & Y \arrow[r, dashed] \arrow[d] 
  & E \arrow[r] \arrow[d] 
  & 0 \\
  & 
  & X \arrow[d] \arrow[r, equal]
  & X \arrow[d] 
  & \\
  & 
  & 0 
  & 0. 
  &
\end{tikzcd}
\]
Not that $X\in \W$ since $\X \subseteq \W$. Then, by the middle column $Y\in \W$ since $\W$ is closed under extensions, Then $Y\in \W\cap\Y$. Thus, $\W\cap\Y$ is cogenerating. \cqfd

\subsection{Hovey triples}\label{Subsec. Hov-triple}  Recall that a Quillen's model structure on a category $\A$ is a triple $({\rm Cof},{\rm Weak}, {\rm Fib})$ of classes of morphisms of $\A$ called cofibrations, weak equivalences, and fibrations, satisfying a set of axioms \cite[Definition 1.1.3.]{Hov99}. These axioms provide a general framework for studying homotopy theory. By this, we mean that we can associate to $\M$ the homotopy category $\Ho_\A(\M)$, which is defined by formally inverting the weak equivalences of $\M$. In other words, $\Ho_\A(\M):=\A[{\rm Weak}^{-1}]$ is obtained after localizing $\A$ at the class of weak equivalences {\rm Weak} (see \cite{Hov99} and  \cite{Sto14} for more details).

In abelian categories, Hovey defined in \cite{Hov02} an abelian model structure as a model structure in the sense of Quillen that satisfies the following two assertions.
\begin{enumerate}
    \item[(i)] A morphism is a cofibration if and only if it is a monomorphism with cofibrant cokernel. 
   \item[(ii)]  A morphism is a fibration if and only if it is an epimorphism with fibrant kernel. 
\end{enumerate}
Then, he showed in \cite[Theorem 2.2]{Hov02} that an abelian model structure on $\A$ is equivalent to a triple $(\Q,\W,\R)$ of classes of objects in $\A$ for which $\W$ is thick (that is, $\W$ is closed under direct summands and satisfies that whenever two out of three terms in a short exact sequence are in $\W$, then so is the third) and $(\Q,\W\cap\R)$ and $(\Q\cap\W,\R)$ are each complete cotorsion pairs. In this case, $\Q$ is precisely the class of cofibrant objects of the model structure, $\R$ is the class of fibrant objects, and $\W$ is the class of trivial objects. We hence denote an abelian model structure $\M$ as a triple $\M=(\Q,\W,\R)$, and for short, we will denote the two associated cotorsion pairs above by $(\Q,\R_\W)$ and $(\Q_\W,\R)$, and often refer to $\M$ as a Hovey triple. We say that $\M$ is hereditary if both the associated cotorsion pairs are hereditary.

\section{Main result} \label{Main result}

In this section, we prove Theorem A from the introduction. The proof is based on Proposition \ref{obj in Qn} and Proposition \ref{obj in Qn 2}; an adjusted version of the dual of \cite[Proposition 3.2]{Elm25} to our situation, in which we significantly improve the given assumptions.
\begin{prop}\label{obj in Qn} Let $n\geq 0$ and $(\Q,\W,\R)$ be a triple of classes of objects such that $(\Q,\R_\W)$ is a complete cotorsion pair. Consider the following assertions for any object $M \in \mathcal{A}$.
	\begin{enumerate}
		\item $M \in \Q_n$.
		\item There is a $\Hom_\A(-,\Q_{\W,n}^\perp)$-exact exact sequence $0\to K\to Q\to M\to 0$ with $Q\in  \Q$ and $K\in \Q_{\W,n-1}$ ($K \in \Q_\W$ if $n=0$).
        \item There is a short exact sequence $0\to M\to Q\to C\to 0$	with $Q\in\Q_{\W,n}$ and $C\in \Q$.
	\end{enumerate}
    Then, (1) $\Leftrightarrow$ (2) $\Rightarrow$ (3). If moreover, $(\Q,\R_\W)$ is hereditary, then (3) $\Rightarrow$ (1).
\end{prop}
\proof
(1) $\Rightarrow$ (3)  We proceed by induction on $n \ge 0$. For $n = 0$, we consider a short exact sequence $
\mathcal{E}: \quad 0 \longrightarrow M \longrightarrow R \longrightarrow C \longrightarrow 0,$
which exists, such that $R \in \R_\W$ and $C \in \Q,$ since $(\Q, \R_\W)$ is complete. We have $M \in \Q$ by assumption,  then  $R \in \Q$ and hence $R \in \Q \cap \W = \Q_{\W}$. Thus, $\mathcal{E}$ is the desired exact sequence.  

Now, assume that (1) $\Rightarrow$ (3) holds for $n-1$ and let us prove that it holds also for $n$. Let $M \in \Q_n$ and consider the following exact sequence $0 \longrightarrow Q_n \longrightarrow \cdots \longrightarrow Q_0 \longrightarrow M \longrightarrow 0$ with $Q_i \in \Q$ for every $i \in \{0, \ldots, n\}$.  Clearly, $Z_0 := \mathrm{\Ker}(Q_0 \to M) \in \Q_{n-1}$.  Then, by hypothesis of induction, there exists an exact sequence $0 \longrightarrow Z_0 \longrightarrow W_0 \longrightarrow C_0 \longrightarrow 0,$
where $W_0 \in \Q_{\W,n-1}$ and $C_0 \in \Q$.  
Consider the following pushout diagram: 
$$
\begin{tikzcd}[column sep=2em, row sep=1em]
&0\arrow[d]&0\arrow[d]&&\\
0 \arrow[r]  
  & Z_0 \arrow[r] \arrow[d] 
  & Q_0 \arrow[r] \arrow[d, dashed] 
  & M \arrow[r] \arrow[d, equal] 
  & 0 \\
0 \arrow[r]  
  & W_0 \arrow[r, dashed] \arrow[d] 
  & E_0 \arrow[r] \arrow[d] 
  & M \arrow[r]  
  & 0 \\ 
  & C_0 \arrow[r,  equal] \arrow[d]
  & C_0 \arrow[d]  &&\\
  & 0  
  & 0.
  &
  &
\end{tikzcd}
$$By the middle column, $E_0 \in \Q$.  Then, by the case $n=0$, we get the following short exact sequence
$0 \longrightarrow E_0 \longrightarrow W_1 \longrightarrow C_1\longrightarrow 0$
with $W_1 \in \Q_{\W}$ and $C_1 \in \Q$.
Consider now another pushout diagram
$$
\begin{tikzcd}[column sep=2em, row sep=1em]
&&0\arrow[d]&0\arrow[d]&\\
0 \arrow[r] 
  & W_0 \arrow[r] \arrow[d, equal] 
  & E_0 \arrow[r] \arrow[d] 
  & M \arrow[r] \arrow[d, dashed] 
  & 0 \\
0 \arrow[r] 
  & W_0 \arrow[r] 
  & W_1 \arrow[r, dashed] \arrow[d] 
  & E_1 \arrow[r] \arrow[d] 
  & 0 \\
  & 
  & C_1 \arrow[d] \arrow[r, equal]
  & C_1 \arrow[d] 
  & \\
  & 
  & 0 
  & 0.
  &
\end{tikzcd}
$$ Note that $E_1 \in \Q_{\W,n}$ by the middle row since $W_1 \in \Q_{\W}$ and $W_0 \in \Q_{\W,n-1}$.
Then,  $0 \longrightarrow M \longrightarrow E_1 \longrightarrow C_1 \longrightarrow 0$ is the desired short exact sequence.

(1) $\Rightarrow$ (2) The case $n=0$:  Since $(\Q,\R_{\W})$ is complete, there exists a short exact sequence: 
$\mathcal{E}: \quad 0 \longrightarrow K \longrightarrow Q \longrightarrow M \longrightarrow 0$ with $Q \in \Q$ and $K \in \R_{\W}$. The sequence $\mathcal{E}$ splits since $M\in \Q$ and $(\Q,\R_{\W})$ is a cotorsion pair; thus, it is $\mathrm{Hom}_{\mathcal{A}}(-,\Q_{\W}^\perp)$-exact. Note that  $K \in \Q$ since $\mathcal{E}$ splits and $Q \in \Q.$  Thus, $K \in \Q \cap \W = \Q_{\W}$. Therefore, $\mathcal{E}$ is the desired sequence.

The case $n \ge 1$: Let $M \in \Q_n$  and consider the following short exact sequence $0 \longrightarrow K_0\longrightarrow Q_0 \longrightarrow M \longrightarrow 0$
which exists, with $Q_0 \in \Q$ and $K_0 \in \R_\W$, since $(\Q,\R_{\W})$ is complete. Note that $K_0 \in \Q_{n-1}$. Then, by (1) $\Rightarrow$(3) there exists a short exact sequence $ 0 \longrightarrow K_0 \longrightarrow Q_1 \longrightarrow C_1\longrightarrow 0$
where $Q_1 \in \Q_{\W,n-1}$ and $C_1 \in \Q$. Now, consider the following pushout diagram:

$$
\begin{tikzcd}[column sep=2em, row sep=1em]
&0\arrow[d]
&0\arrow[d]
&
&\\
0 \arrow[r]  
& K_0 \arrow[r] \arrow[d] 
& Q_0 \arrow[r] \arrow[d, dashed] 
& M \arrow[r] \arrow[d, equal] 
& 0 \\
0 \arrow[r]  
& Q_1 \arrow[r, dashed] \arrow[d] 
& E_0 \arrow[r] \arrow[d] 
& M \arrow[r]  
& 0 
\\ 
& C_1 \arrow[r,  equal] \arrow[d]
& C_1 \arrow[d]  
&
&
\\
& 0  
& 0.
&
&
\end{tikzcd}
$$
From the middle column, $E_0 \in \Q$. Now, let us prove that the  short exact sequence $0\to Q_1\to E_0\to M\to 0$ is $\Hom_\A(-,\Q_{\W,n}^\perp)$-exact. Since $(\Q,\R_{\W})$ is  complete, there is an exact sequence $
0 \longrightarrow E_0 \longrightarrow R \longrightarrow C_2\longrightarrow 0$ with $R\in \R_\W$ and $C_2 \in \Q$.  
Consider another pushout:

$$
\begin{tikzcd}[column sep=2em, row sep=1em]
&&0\arrow[d]&0\arrow[d]&\\
0 \arrow[r] 
  & Q_1 \arrow[r,,"h"]\arrow[d, equal] 
  & E_0 \arrow[r] \arrow[d,,"g"] 
  & M \arrow[r] \arrow[d, dashed] 
  & 0 \\
0 \arrow[r] 
  & Q_1 \arrow[r,"f"]
  & R \arrow[r, dashed] \arrow[d] 
  & E_1 \arrow[r] \arrow[d] 
  & 0 \\
  & 
  & C_2 \arrow[d] \arrow[r, equal]
  & C_2 \arrow[d] 
  & \\
  & 
  & 0 
  & 0.
  &
\end{tikzcd}
$$
By the middle column, $R\in \Q$, then $R\in\Q_{\W}=\Q \cap \W$ since $R\in\R_{\W} \subseteq \W$. Then, if we consider the middle row we get $E_1\in\Q_{\W,n}$ since $Q_1\in\Q_{\W,n-1}$.

Now, let $Y\in{}\Q_{\W,n}^{\perp}$. Then, $\operatorname{Hom}_{\mathcal{A}}(f,Y):\operatorname{Hom}_{\mathcal{A}}(R,Y) \to \operatorname{Hom}_{\mathcal{A}}(Q_1,Y)$ is epic by the middle row since $\operatorname{Ext}^{1}_{\mathcal{A}}(E_1,Y)=0.$
Then, $\operatorname{Hom}_{\mathcal{A}}(h,Y):\operatorname{Hom}_{\mathcal{A}}(E_0,Y)\to \operatorname{Hom}_{\mathcal{A}}(Q_1,Y)$ is also epic since $\operatorname{Hom}_{\mathcal{A}}(f,Y)=\operatorname{Hom}_{\mathcal{A}}(h,Y) \operatorname{Hom}_{\mathcal{A}}(h,Y)$. That is, the sequence $0\to Q_1\to E_0\to M\to 0$ is $\operatorname{Hom}_{\mathcal{A}}(-,Y)$-exact and hence it is the desired short exact sequence.
 
(2) $\Rightarrow$ (1) If $n = 0$, let $R  \in \R_W$ and $ \mathcal{E}: 0\to K\to Q\to M\to 0$ be a $\Hom_\A(-,\Q_{\W}^\perp)$-exact short exact sequence  with $Q\in  \Q$ and $K\in \Q_{\W}$. Since $\R_W={\Q}^\perp \subseteq{\Q_\W}^\perp$, $\mathcal{E}$ is $\Hom_\A(-,R)$-exact. That is, $0 \to \Ext_\A(M,R) \to \Ext_\A(Q,R) \to \cdots$ is exact. But, $\Ext_\A(Q,R)=0$, then $\Ext_\A(M,R)=0$. Thus, $M \in {}^\perp\R_\W=\Q$.

If $n \ge 1$, let  $0\to K\to Q\to M\to 0$ be a short exact sequence such that $Q\in  \Q$ and $K\in \Q_{\W,n-1}$. But, $\Q_\W \subseteq \Q$, then $K\in \Q_{n-1}$ which implies that $M\in \Q_n.$

(3) $\Rightarrow$ (1)  Let $0\to M\to Q\to C\to 0$ be a short exact sequence such that $Q\in\Q_{\W,n}$ and $C\in \Q$. Then, $Q\in\Q_{n}$ since $\Q_{\W} \subseteq \Q$. Thus, $M \in \Q_n$ by  \cite[Proposition 5.2(1)]{BEGO25} since $(\Q,\R_\W)$ is a complete hereditary cotorsion pair.
\cqfd

\begin{cor}\label{cor obj in Qn} Let $n\geq 0$ and $(\Q,\W,\R)$ be a triple of classes of objects such that $(\Q,\R_\W)$ is a complete cotorsion pair and $\W$ is coresolving. Then, $\Q_{\W,n}= \W\cap\Q_n.$ Consequently, for all $M\in \W$, $\resdim_{\Q_{\W}} M=\resdim_{\Q} M$.
\end{cor}
\proof
We have $\Q_{\W,n} \subseteq \W_n\cap\Q_n$ since $\Q_\W \subseteq \Q$ and $\Q_\W \subseteq \W$. But, $\W_n \subseteq \W$ since $\W$ is closed under cokernels of monomorphisms. Thus, $\Q_{\W,n} \subseteq \W\cap\Q_n.$

Conversely, let $M\in\W\cap\Q_n$ and consider a short exact sequence  $0\to K\to Q\to M\to 0$ which exists, such that $Q \in \Q$ and $K \in \Q_{\W,n-1}$, by Proposition \ref{obj in Qn}. Since $\W$ is closed under cokernels of monomorphisms $K \in \W$. Then, $Q \in \W$ since $M \in\W$ and $\W$ is closed under extensions. Thus, $Q \in \Q_\W$. Therefore, $M\in \Q_{\W,n}$.
\cqfd
\begin{prop}\label{obj in Qn 2} Let $n\geq 0$ and $(\Q,\W,\R)$ be a triple of classes of objects such that $(\Q,\R_\W)$ is a complete hereditary cotorsion pair, $\Q_\W^\perp\subseteq\R$, $(\Q_{\W,n},\Q_{\W,n}^\perp)$ is a left complete cotorsion pair and $\W$ is closed under kernels of epimorphisms. Then, the following assertions are equivalent for any object $M \in \A$:

\begin{enumerate}
		\item[$(1)$] $M \in \Q_n$.
		\item[$(2)$] $\Ext^1_\A(M,C)=0$ for all $C\in \W\cap \Q_{\W,n}^\perp$.
	\end{enumerate}

    In this case, $( \Q_n,\W\cap \Q_{\W,n}^\perp)$ is a hereditary cotorsion pair.
\end{prop}
\proof 
$(1) \Rightarrow (2)$ Let $M \in \Q_n$ and $C\in \W\cap\Q_{\W,n}^{\perp}$. By Proposition \ref{obj in Qn} there exists a $\operatorname{Hom}_{\mathcal{A}}({}{-},\Q_{\W,n}^{\perp})$-exact exact sequence $0\to K\stackrel{\varphi}{\longrightarrow} Q\to M\to 0$ with $Q\in\Q$ and $K \in \Q_{\W,n-1}$. Then,  $\operatorname{Hom}_{\mathcal{A}}(\varphi,C)$ is epic, which gives rise to the following exact sequence of abelian groups $
0\to \operatorname{Ext}^{1}_{\mathcal{A}}(M,C)\to\operatorname{Ext}^{1}_{\mathcal{A}}(Q,C)\to \cdots \textit{\, .}$
We have $\operatorname{Ext}^{1}_{\mathcal{A}}(Q,C)=0$ since \(Q\in\Q\),
and $C\in\Q_{\W,n}^{\perp}\cap\W \subseteq \Q_{\W}^{\perp}\cap\W \subseteq {\R}\cap{\W}={\R}_{\W}={\Q^\perp}.$
Thus, we get \(\operatorname{Ext}^{1}_{\mathcal{A}}(M,C)=0\).

 $(2) \Rightarrow (1)$ Since the cotorsion pair $(\Q, \R_\W)$ is complete, there is a short exact sequence $0 \to M \to R \to C \to 0$ with $R \in \R_\W$ and $C \in \Q$.  Let $ 0 \to K \to Q \to R \to 0 $ be a short exact sequence which exists such that $ Q \in \Q_{\W,n}$ and $K \in \Q_{\W,n}^\perp$ since $(\Q_{\W,n},\Q_{\W,n}^\perp)$ is left complete. Now, we consider the following pullback diagam:
$$
\begin{tikzcd}[column sep=2em, row sep=1em] 
&0\arrow[d]&0\arrow[d]&&\\
 & K \arrow[r,  equal] \arrow[d]
  & K \arrow[d]  &&\\
0 \arrow[r]  
  & D \arrow[r, dashed] \arrow[d, dashed] 
  & Q \arrow[r] \arrow[d] 
  & C \arrow[r] \arrow[d, equal] 
  & 0 \\
0 \arrow[r]  
  & M \arrow[r] \arrow[d] 
  & R\arrow[r] \arrow[d] 
  & C\arrow[r]  
  & 0 \\ 
  & 0  
  & 0.&&
\end{tikzcd}
$$
Since $R,Q \in \W$ and $\W$ is closed under kernels of epimorphisms, $K \in \W$, that is, $K \in \W \cap \Q_{\W,n}^\perp$. Then, the first column splits since $\Ext^1_\A(M,K)=0$ by assumption. Thus, $M$ is a direct summand of $D$. But, $D\in \Q_n$ by the middle row since $C \in \Q$ and $Q\in \Q_{\W,n} \subseteq \Q_n$ (see \cite[Proposition 5.2]{BEGO25}). Thus, $M\in \Q_n$.\\

 Now, let us prove that $( \Q_n,\W\cap \Q_{\W,n}^\perp)$ is a hereditary cotorsion pair. To do so, it suffices to prove  $\Q_n^\perp = \W \cap \Q_{\W,n}^\perp$ since $\Q_n={}^\perp (\W \cap \Q_{\W,n}^\perp)$ by $(1) \Leftrightarrow (2)$ and $Q_n$ is closed under kernels of epimorphisms by \cite[Proposition 5.2(1)]{BEGO25}). We have $\Q\subseteq \Q_n$ and $\Q_{\W,n}\subseteq \Q_n$, hence  $\Q \cup \Q_{\W,n}\subseteq \Q_n.$ Then, $\Q_n^\perp \subseteq (\Q \cup \Q_{\W,n})^\perp=\Q^\perp \cap \Q_{\W,n}^\perp=\R_\W \cap \Q_{\W,n}^\perp.$ Thus, $\Q_n^\perp \subseteq \W \cap \Q_{\W,n}^\perp$ since $\R_W=R\cap \W \subseteq \W.$
 On the other hand, we have already shown that $\Q_n={}^\perp(\W\cap\Q_{W,n}^\perp)$. Then, $\W\cap\Q_{W,n}^\perp\subseteq({}^\perp(\W\cap\Q_{W,n}^\perp))^\perp= \Q_n^\perp$. Therefore, $\Q_n^\perp = \W \cap \Q_{\W,n}^\perp$. \cqfd 

\begin{rem}
  Each result in this section has a dual version with a completely dual proof. We leave the details to the reader.  
\end{rem}

Now, we are in a position to prove Theorem A from Introduction.\\
\\
\textbf{Proof of Theorem A.} We only prove assertion (a) as assertion (b) is dual.

We have $\W$ is thick since $\M=(\Q,\W,\R)$ is a Hovey triple, $(\Q_n, \W \cap \Q_{\W,n}^\perp)$ is a cotorsion pair by Proposition \ref{obj in Qn 2} and $(\Q_n \cap\W, \Q_{\W,n}^\perp) = (\Q_{\W,n} , \Q_{\W,n}^\perp)$ is a complete cotorsion pair by assumption ($\Q_n \cap\W=\Q_{\W,n}$ by Corollary \ref{cor obj in Qn}).

 Now, let us prove that $(\Q_n, \W \cap \Q_{\W,n}^\perp)$ is complete. To this end, let $M \in \A$ and consider the following short exact sequences $0 \to K_0 \to Q_0 \to M \to 0$ and $0 \to K_0 \to Q_1 \to C_1 \to 0.$
 The first sequence exists such that $Q_0\in \Q$ and $K_0 \in \R_W$ since $(\Q,\R_\W)$ is complete and the second one exists such that $Q_1\in \Q_{\W,n}^\perp$ and $C_1 \in \Q_{\W,n}$ since $(\Q_{\W,n},\Q_{\W,n}^\perp)$ is complete. Now, we consider the following pushout diagam:
\[
\begin{tikzcd}[column sep=2em, row sep=1em] 
&0\arrow[d]&0\arrow[d]&&\\
0 \arrow[r]  
  & K_0 \arrow[r] \arrow[d] 
  & Q_0 \arrow[r] \arrow[d] 
  & M \arrow[r] \arrow[d, equal] 
  & 0 \\
0 \arrow[r]  
  & Q_1 \arrow[r] \arrow[d] 
  & E_1 \arrow[r] \arrow[d] 
  & M \arrow[r]  
  & 0 \\ 
  & C_1 \arrow[r,  equal] \arrow[d]
  & C_1 \arrow[d]  &&\\
  & 0  
  & 0.&&
\end{tikzcd}
\]
The middle column gives $E_1 \in \Q_n$ by \cite[Proposition 5.2(1)]{BEGO25} since $Q_0\in \Q$ and $C_1 \in \Q_{\W,n}\subseteq \Q_n$. The first column gives $Q_1 \in \W$ since $K_0\in \R_\W \subseteq \W$ and $C_1 \in \Q_{\W,n} \subseteq \W$ ($\W$ is closed under cokernels of monomorphisms). Thus, $Q_1 \in \W \cap \Q_{\W,n}^\perp.$  Therefore, $(\Q_n, \W\cap\Q_{\W,n}^\perp)$ is left complete. 

The right completeness of $(\Q_n, \W\cap\Q_{\W,n}^\perp)$ is a consequence of Lemma \ref{lem-left-right-complete} and Lemma \ref{lem-intersection-generating}. Indeed, $(\Q_{\W,n},\Q_{\W,n}^\perp)$ is right complete by assumption, $\W$ is cogenerating ( since $(\Q,\R_\W)$ is right complete and $\R_\W\subseteq \W$) and $\Q_{\W,n}\subseteq \W$ since $\Q_\W \subseteq \W$ and $\W$ is closed undeer cokernels of monomorphisms. Then, $\W\cap\Q_{\W,n}^\perp$ is cogenerating by Lemma \ref{lem-intersection-generating}. On the other hand, we have already proved that $(\Q_n, \W\cap\Q_{\W,n}^\perp)$ is left complete, then it is also right complete by Lemma \ref{lem-left-right-complete}.

 Now, let us prove that $\M_n=(\Q_n, \W, \Q_{\W,n}^\perp)$ is hereditary.  Since $M=(\Q,\W,\R)$ is a hereditary Hovey triple, $(\Q,\R_\W)$ and $(\Q_\W,\R)$ are complete hereditary cotorsion pairs. Then, by \cite[Proposition 5.2(1)]{BEGO25}, $\Q_n$ and $\Q_{\W,n}$ are closed under kernels of epimorphisms. Thus, $(\Q_n, \W\cap\Q_{\W,n}^\perp)$ and $(\Q_{\W,n}, \Q_{\W,n}^\perp)$ are hereditary. Therefore,  $\M_n=(\Q_n, \W, \Q_{\W,n}^\perp)$ is a hereditary Hovey triple.
\cqfd

\section{Application in the category of sheaves}

In this section, \(X\) denotes a scheme with structure sheaf \(\mathcal{O}_X\). By a sheaf on \(X\) we always mean a quasi-coherent sheaf of \(\mathcal{O}_X\)-modules. Our aim is to apply the results of the previous section to the Gorenstein flat model structure constructed in \cite{CET21} and the Gorenstein injective model structure constructed in \cite{EG25} on \(\Qcoh(X)\), the category of quasi-coherent sheaves on \(X\).

Throughout this section, we assume that \(X\) is semi-separated and quasi-compact. For most of the paper, we will strengthen this hypothesis by assuming further that \(X\) is Noetherian (i.e. locally Noetherian and quasi-compact).

Recall that a sheaf $F$ on $X$ is called flat if, for every open affine subset $U \subseteq X$, the $\mathcal{O}_X(U)$-module $F(U)$ is flat. Equivalently, $F$ is flat if and only if the tensor product functor $-\otimes F$ is exact. Here $-\otimes - \;:=\; -\otimes_{\mathcal{O}_X} -$ denotes the usual tensor product on $\mathcal{O}_X\text{-Mod}$, the category of all sheaves of $\mathcal{O}_X$-modules (see e.g.\ \cite[Section 7.4]{GW10}).

Denote by $\flat(X)$ the class of flat sheaves on $X$. Estrada and Enochs showed \cite[Corollary 4.2]{EE05} that $\flat(X)$ is a covering class. On the other hand, the category $\Qcoh(X)$ is a Grothendieck abelian category, hence it has a generator. By \cite[Lemma A.1]{EP15} it follows that $\Qcoh(X)$ admits a flat generator. Consequently, $\flat(X)$ is a  special precovering class, i.e., every sheaf has a special flat precover. Therefore, the pair $(\flat(X),\cot(X))$ is a complete cotorsion pair where $\cot(X) =\flat(X)^\perp$ denotes the class of cotorsion sheaves. Moreover, by \cite[Proposition 6.4]{Gil07} this cotorsion pair is cogenerated by a set.
 
In the following result, we extend this fact to \(\flat_n(X)\), the class of all sheaves with flat dimension at most an integer \(n\geq 0\).

\begin{prop}\label{n-flat-cot-pair} 
Assume that \(X\) is a quasi-compact and semi-separated scheme. Then the pair
\[
\bigl(\flat_n(X),\ \flat_n(X)^\perp\bigr)
\]
is a complete and hereditary cotorsion pair cogenerated by a set.
\end{prop}

\proof We first show that for any integer \(n\geq 1\) and any sheaf \(M\) on \(X\), 
\[
M\in\flat_n(X)\quad\Longleftrightarrow\quad M(U)\in\flat_n\bigl(\mathcal{O}_X(U)\bigr) \text{ for every open affine }U\subseteq X .
\]
Where $\flat_n\bigl(\mathcal{O}_X(U)\bigr)$ denotes the class of all $\mathcal{O}_X(U)$-modules with flat dimension at most an integer \(n\geq 0\).

Let \(M\) be a sheaf. By \cite[Corollary 4.2]{EE05}, there exists a flat resolution
\[
\cdots \to F_{2}\to F_1 \to F_0 \to M \to 0,
\]
in which each syzygy \(K_{k}:=\Ker(F_{k}\to F_{k-1})\) is cotorsion. For every open affine subset \(U\subseteq X\) this induces a flat resolution of \(M(U)\):
\[
\cdots \to F_2(U)\to F_1(U)\to F_0(U)\to M(U)\to 0.
\]
Hence, keeping in mind \cite[Proposition 5.1]{BEGO25}), we have the equivalences
\begin{align*}
M\in\flat_n(X) 
   &\Longleftrightarrow K_{n-1}\in\flat(X)  
      \\
   &\Longleftrightarrow K_{n-1}(U)\in\flat\bigl(\mathcal{O}_X(U)\bigr) 
      \qquad\text{for every open affine }U\subseteq X\\
   &\Longleftrightarrow M(U)\in\flat_n\bigl(\mathcal{O}_X(U)\bigr) 
      \qquad\text{for every open affine }U\subseteq X .
\end{align*}

Now, let $\mathbf{Q}_X=(\textbf{V},\textbf{E})$ be the quiver defined with the set $\textbf{V}$ of vertices, all open affine subsets of $X$, and the set of edges $\textbf{E}$ consists of the reversed arrows $(V \to U)$ corresponding to the inclusions $U\subseteq V$, where $U$ and $V$ are open affine subsets of $X$. By \cite[Theorem 3.4(2)]{DM07}, for every open affine $U\subseteq X$, the pair $$(\flat_n(\O_X(U),\flat_n(\O_X(U)^\perp)$$
is a complete cotorsion pair cogenerated by a set $\mathcal{S}_U$. Applying \cite[Corollary 3.15]{EAP12}, we obtain that
$$(\flat_n(X),\flat_n(X)^\perp)$$
is a complete cotorsion pair cogenerated by a (representative) set of the class
$$\mathcal{L}:=\{ L\in \Qcoh(X)|\,\, L(U)\in \mathcal{S}_U \text{ for every open affine subset $U\subseteq X$ } \}.$$

Finally, since \((\flat(X),\flat(X)^\perp)\) is a complete hereditary cotorsion pair, it follows from \cite[Proposition 5.2(1)]{BEGO25} that \(\flat_n(X)\) is closed under kernels of epimorphisms; consequently, the cotorsion pair \((\flat_n(X),\flat_n(X)^\perp)\) is hereditary. \cqfd

Recall that a sheaf $M$ is called Gorenstein flat \cite[Definition 1.2]{CET21} if there exists an acyclic complex of flat sheaves $\textbf{F}$ with  $M \cong Z_0(\textbf{F})$  the zero cycle of $\textbf{F}$ such that the complex $E\otimes \textbf{F}$ is acyclic for every injective sheaf $E$.
The class of all Gorenstein flat sheaves is denoted $\GF(X)$. Over a semi-separated Noetherian scheme $X$, Christensen, Estrada, and Thompson proved \cite[Theorem 2.5]{CET21} that there exists on $\Qcoh(X)$ a unique hereditary abelian model structure given by
$$\mathcal{M}_{\GF}=(\GF(X),W_{\mathrm{flat}} ,\flat(X)^\perp),$$
called the \textbf{Gorenstein flat model structure}.  

Now, we use Theorem A to extend this model structure to the setting of sheaves with bounded Gorenstein flat dimension.

\begin{thm}[$n$-Gorenstein flat model structure over $X$] \label{n-GF-Hovey-triple} 
Assume that \(X\) is semi-separated Noetherian. Then there exists a hereditary abelian model structure
\[
\mathcal{M}_{\GF_n}= \bigl( \GF_n(X),\ \mathcal{W}_{\mathrm{flat}},\ \flat_n(X)^\perp \bigr)
\]
on \(\Qcoh(X)\). Consequently, the pair \((\GF_n(X),\GF_n(X)^\perp)\) is a complete hereditary cotorsion pair with
$
\GF_n(X)^\perp = \mathcal{W}_{\mathrm{flat}} \cap \flat_n(X)^\perp, $
and \(\GF_n(X)\) is a special precovering class.
\end{thm}

\proof Apply Theorem A and Proposition \ref{n-flat-cot-pair} to the Gorenstein flat model structure. \cqfd 

\begin{rem}
A similar argument to that of Proposition~\ref{n-flat-cot-pair} shows that the pair 
\((\GF_n(X),\GF_n(X)^\perp)\) is a hereditary complete cotorsion pair. 
On the other hand, Proposition~\ref{n-flat-cot-pair} already tells us that 
\((\flat_n(X),\flat_n(X)^\perp)\) is also a hereditary complete cotorsion pair. Moreover, these two cotorsion pairs have the same core; that is,
\[
\GF_n(X) \cap \GF_n(X)^\perp = \flat_n(X) \cap \flat_n(X)^\perp .
\]
This equality follows by an argument analogous to \cite[Lemma 2.3]{CET21}. Hence, by \cite[Theorem 1.1]{Gil15} there exists a unique thick class \(\mathcal{W}_{\mathrm{flat},n}\) such that
\[
\mathcal{M}_{\GF_n} = \bigl( \GF_n(X),\ \mathcal{W}_{\mathrm{flat},n},\ \flat_n(X)^\perp \bigr)
\]
is a hereditary Hovey triple on \(\Qcoh(X)\). 

So Theorem~\ref{n-GF-Hovey-triple} shows that 
\(\mathcal{W}_{\mathrm{flat},n}\) does not depend on the integer $n$ and  coincides with the class \(\mathcal{W}_{\mathrm{flat}}\) of trivial objects in the 
Gorenstein flat model structure. In terms of homotopy categories, we obtain an equivalence of 
triangulated categories
$$
\Ho(\mathcal{M}_{\GF_n}) \simeq \Ho(\mathcal{M}_{\GF}) \hspace{0.5cm}\forall n\geq 0. $$

\end{rem}

\bigskip

Besides Theorem~\ref{n-GF-Hovey-triple}, other results of interest follow from the previous section. First, we obtain a characterization of sheaves with bounded Gorenstein flat dimension.

\begin{prop} \label{n-GF sheaf}  
Assume that \(X\) is a semi-separated Noetherian scheme. The following assertions are equivalent for any sheaf \(M\):
\begin{enumerate}
    \item \(\Gfd_X(M)\leq n\).
    \item There exists a \(\Hom_X(-,\flat_n(X)^\perp)\)-exact exact sequence of sheaves
          \[ 0\to K\to P\to M\to 0 \]
          where \(P\) is Gorenstein flat and \(\fd_X(K)\leq n-1\) (if \(n=0\), \(K\) is flat).
    \item \(\Ext^1_X(M,C)=0\) for every \(C\in W_{\mathrm{flat}}\cap \flat_n(X)^\perp\).
    \item There exists a short exact sequence of sheaves
          \[ 0\to M\to E\to L\to 0 \]
          where \(\fd_X(E)\leq n\) and \(L\) is Gorenstein flat.
\end{enumerate}
Consequently, $\flat_n(X) = W_{\mathrm{flat}} \cap \GF_n(X). $
\end{prop}

\proof 
Apply Propositions~\ref{obj in Qn}, \ref{obj in Qn 2} and~\ref{n-flat-cot-pair} to the Gorenstein flat model structure.
\cqfd

As in the affine case, it is natural to ask how much the flat dimension differs from the Gorenstein flat dimension. As an immediate consequence of Proposition~\ref{n-GF sheaf}, we obtain a non-affine answer (compare with the affine case \cite[Proposition 3.8]{Emm24} or \cite[Theorem 3.4]{Elm24}).

\begin{cor} \label{GF-ref}
Assume that \(X\) is semi‑separated Noetherian. If \(M\) is a sheaf on \(X\), then  
\(\Gfd_X(M)\leq \fd_X(M)\), with equality when \(M\in W_{\mathrm{flat}}\).  
In particular, equality holds when \(M\) has finite flat or finite injective dimension.
\end{cor}

\proof
The inequality \(\Gfd_X(M)\leq \fd_X(M)\) follows since every flat sheaf is Gorenstein flat. The equality when \(M\in W_{\mathrm{flat}}\) is a direct consequence of Proposition~\ref{n-GF sheaf}.

Finally, Theorem~\ref{n-GF-Hovey-triple} gives the equalities
\[
\GF(X)^\perp = W_{\mathrm{flat}} \cap \flat(X)^\perp
\qquad\text{and}\qquad
\flat_n(X) = W_{\mathrm{flat}} \cap \GF_n(X).
\]
From these, we see that the class \(W_{\mathrm{flat}}\) contains all injective sheaves and all flat sheaves. Since \(W_{\mathrm{flat}}\) is thick, it also contains every sheaf of finite injective or finite flat dimension.
\cqfd

\bigskip
Now we turn our attention to the dual situation.

\begin{lem} 
Assume that \(X\) is a quasi-compact and semi-separated scheme. Then
\[
\bigl( {}^\perp \inj_n(X),\ \inj_n(X) \bigr)
\]
is a complete and hereditary cotorsion pair.
\end{lem}

\proof
See \cite[Propositions 6.2.13 and 8.2.2]{Per16}.
\cqfd

Recall from \cite[Definition 3.1]{EI17} that a sheaf \(M\) is called \emph{Gorenstein injective} if it is the zero cycle \(M \cong Z_0(\mathbf{E})\) of an acyclic complex \(\mathbf{E}\) of injective sheaves such that the complex \(\Hom_X(E,\mathbf{E})\) is acyclic for every injective sheaf \(E\). We denote by \(\GI(X)\) the class of all Gorenstein injective sheaves.

Based on work of Estrada and Iacob \cite[Lemmas 3.4 and 3.5]{EI17}, which in turn traces back to Krause \cite[Theorem 7.12]{Kra05}, it was recently proved by Estrada and Gillespie \cite[Theorem 5.2]{EG25} that for a semi‑separated Noetherian scheme \(X\) there exists a hereditary abelian model structure given by 
\[\M_{\GI}=
\bigl( {}^\perp\!\inj(X),\ \mathcal{W}_{\mathrm{inj}},\ \GI(X) \bigr)
\]
on \(\Qcoh(X)\). This model structure is called the \textbf{Gorenstein injective model structure}. In this case, the class of trivial objects is described by $
\mathcal{W}_{\mathrm{inj}} = {}^\perp\!\GI(X).$

The following result can be seen as a dual version of Theorem~\ref{n-GF-Hovey-triple}.

\begin{thm}[$n$-Gorenstein injective model structure] \label{GIn-cot-pair-Hov-trip}
Assume that \(X\) is semi‑\\separated Noetherian. Then there exists a hereditary abelian model structure
\[ \M_{\GI_n}=
\bigl( {}^\perp\!\inj_n(X),\ \mathcal{W}_{\mathrm{inj}},\ \GI_n(X) \bigr)
\]
on \(\Qcoh(X)\). Consequently, the pair \(\bigl({}^\perp\!\GI_n(X),\GI_n(X)\bigr)\) is a complete hereditary cotorsion pair with
$
{}^\perp\!\GI_n(X) = {}^\perp\!\inj_n(X) \cap \mathcal{W}_{\mathrm{inj}},$
and \(\GI_n(X)\) is a special preenveloping class.
\end{thm}

\begin{prop} \label{n-GI sheaf}  Assume $X$ is a  semi-separated Noetherian scheme. The following assertions are equivalent for any sheaf $M$:
\begin{enumerate}
    \item $\Gid_X(M)\leq n.$

 \item There exists a $\Hom_X(\,^\perp\inj_n(X),-)$-exact exact sequence of sheaves $$0\to M\to E\to L\to 0$$
where $L\in \inj_{n-1}(X)$ and $E\in \GI(X)$. If $n=0$, this is understood to mean $L\in \inj(R)$.
\item $\Ext^1_X(C,M)=0$ for every sheaf  $C\in \,^\perp \inj_n(X) \cap {}^\perp\GI(X)$.

\item There exists an exact sequence of sheaves $$0\to K\to E\to M\to 0$$
where $E\in\inj_n(X)$ and $K\in \GI(X)$.
	
\end{enumerate}

 Consequently, $\inj_n(X)={}^\perp\GI(X)\cap \GI_n(X).$
\end{prop}

In the following, we show that the Gorenstein injective dimension of sheaves is a refinement of the usual injective dimension.

Unlike the affine case (see \cite[Corollary 3.5(1)]{Elm25}), the fact that the class \(W_{\mathrm{inj}} = {}^\perp\!\GI(X)\) contains all flat sheaves (and therefore all sheaves with finite flat dimension) cannot be deduced from closure under direct limits because flat sheaves are not generally direct limits of projective sheaves. Instead, we overcome this obstacle by invoking a recent non-trivial result of Christensen, Estrada, and Thompson \cite[Theorem 3.3]{CET21}: every acyclic complex of cotorsion sheaves has cotorsion cycles.

\begin{cor} \label{GI-ref}
Assume that \(X\) is a semi‑separated Noetherian scheme. If \(M\) is a sheaf on \(X\), then
$
\Gid_X(M)\le \id_X(M),$ 
and equality holds whenever \(M\in {}^\perp\!\GI(X)\). In particular, equality holds when \(M\) has finite injective or finite flat dimension.
\end{cor}

\proof  The argument follows that of Corollary~\ref{GF-ref}, except for one key point: we must justify why \({}^\perp\!\GI(X)\) contains flat sheaves (and hence sheaves with finite flat dimension). This is equivalent to showing that every Gorenstein injective sheaf is cotorsion.

Let \(M\) be a Gorenstein injective sheaf. By definition, \(M\cong Z_0(\mathbf{E})\) for some acyclic complex \(\mathbf{E}\) of injective sheaves. Every injective sheaf is cotorsion, so \cite[Theorem 3.3]{CET21} implies that every cycle of \(\mathbf{E}\) is cotorsion; in particular, \(M\) is cotorsion.  \cqfd

 Unlike the affine case (see \cite[pages 22 and 23]{SS20}) and \cite[Theorem B]{Elm24}), the class $W_{\mathrm{flat}} $ of trivial objects in the ($n$-)Gorenstein flat model structure is not yet well described.  We conclude this article with a partial answer to this problem, due mainly to Estrada and Gillespie \cite{EG25}.

 Recall that a commutative Noetherian ring is called (Iwanaga--)Gorenstein if it has finite self-injective dimension. A scheme $X$ is Gorenstein provided that $\O_{X,x}$ is a Gorenstein ring for every $x\in X$.

\begin{cor} \label{Gorenstein-trivial-class}
Assume that \(X\) is a semi-separated Gorenstein scheme of finite Krull dimension \(d\). Then
$
W_{\mathrm{flat}} = \flat_d(X) = \inj_d(X) = W_{\mathrm{inj}}.$ 
Consequently, for every \(n\ge 0\) there are equivalences of triangulated categories
\[
\operatorname{Ho}\bigl(\mathcal{M}_{\GF_n}\bigr) \simeq
\operatorname{Ho}\bigl(\mathcal{M}_{\GF}\bigr) \simeq
\operatorname{Ho}\bigl(\mathcal{M}_{\GI}\bigr) \simeq
\operatorname{Ho}\bigl(\mathcal{M}_{\GI_n}\bigr).
\]
\end{cor}
\proof 
The equalities follow from \cite[Proposition 9.5 and the remark following it]{EG25}. Hence the (\(n\)-)Gorenstein flat model structure and the (\(n\)-)Gorenstein injective model structure on \(\Qcoh(X)\) share the same class of trivial objects; therefore, their homotopy categories are equivalent as triangulated categories.  \cqfd

\end{document}